\documentclass[twoside,leqno,10pt, A4]{amsart}
\usepackage{amsfonts}
\usepackage{amsmath}
\usepackage{amscd}
\usepackage{amssymb}
\usepackage{amsthm}
\usepackage{amsrefs}
\usepackage{latexsym}
\usepackage{mathrsfs}
\usepackage{bbm}
\usepackage{enumerate}
\usepackage{graphicx}

\usepackage{amsfonts}
\usepackage{amsmath}
\usepackage{amscd}
\usepackage{amssymb}
\usepackage{amsthm}
\usepackage{amsrefs}
\usepackage{latexsym}
\usepackage{mathrsfs}
\usepackage{bbm}
\usepackage{amscd}
\usepackage{amssymb}
\usepackage{amsthm}
\usepackage{amsrefs}
\usepackage{latexsym}
\usepackage{mathrsfs}
\usepackage{bbm}
\usepackage{enumerate}
\usepackage{graphicx}
\usepackage{color}
\begin{document}

\newtheorem{theorem}[subsection]{Theorem}
\newtheorem{proposition}[subsection]{Proposition}
\newtheorem{lemma}[subsection]{Lemma}
\newtheorem{corollary}[subsection]{Corollary}
\newtheorem{conjecture}[subsection]{Conjecture}
\newtheorem{prop}[subsection]{Proposition}
\numberwithin{equation}{section}
\newcommand{\mr}{\ensuremath{\mathbb R}}
\newcommand{\mc}{\ensuremath{\mathbb C}}
\newcommand{\dif}{\mathrm{d}}
\newcommand{\intz}{\mathbb{Z}}
\newcommand{\ratq}{\mathbb{Q}}
\newcommand{\natn}{\mathbb{N}}
\newcommand{\comc}{\mathbb{C}}
\newcommand{\rear}{\mathbb{R}}
\newcommand{\prip}{\mathbb{P}}
\newcommand{\uph}{\mathbb{H}}
\newcommand{\fief}{\mathbb{F}}
\newcommand{\majorarc}{\mathfrak{M}}
\newcommand{\minorarc}{\mathfrak{m}}
\newcommand{\sings}{\mathfrak{S}}
\newcommand{\fA}{\ensuremath{\mathfrak A}}
\newcommand{\mn}{\ensuremath{\mathbb N}}
\newcommand{\mq}{\ensuremath{\mathbb Q}}
\newcommand{\half}{\tfrac{1}{2}}
\newcommand{\f}{f\times \chi}
\newcommand{\summ}{\mathop{{\sum}^{\star}}}
\newcommand{\chiq}{\chi \bmod q}
\newcommand{\chidb}{\chi \bmod db}
\newcommand{\chid}{\chi \bmod d}
\newcommand{\sym}{\text{sym}^2}
\newcommand{\hhalf}{\tfrac{1}{2}}
\newcommand{\sumstar}{\sideset{}{^*}\sum}
\newcommand{\sumprime}{\sideset{}{'}\sum}
\newcommand{\sumprimeprime}{\sideset{}{''}\sum}
\newcommand{\sumflat}{\sideset{}{^\flat}\sum}
\newcommand{\shortmod}{\ensuremath{\negthickspace \negthickspace \negthickspace \pmod}}
\newcommand{\V}{V\left(\frac{nm}{q^2}\right)}
\newcommand{\sumi}{\mathop{{\sum}^{\dagger}}}
\newcommand{\mz}{\ensuremath{\mathbb Z}}
\newcommand{\leg}[2]{\left(\frac{#1}{#2}\right)}
\newcommand{\muK}{\mu_{\omega}}
\newcommand{\thalf}{\tfrac12}
\newcommand{\lp}{\left(}
\newcommand{\rp}{\right)}
\newcommand{\Lam}{\Lambda_{[i]}}
\newcommand{\lam}{\lambda}
\def\L{\fracwithdelims}
\def\om{\omega}
\def\pbar{\overline{\psi}}
\def\phi{\varphi}
\def\lam{\lambda}
\def\lbar{\overline{\lambda}}
\newcommand\Sum{\Cal S}
\def\Lam{\Lambda}
\newcommand{\sumtt}{\underset{(d,2)=1}{{\sum}^*}}
\newcommand{\sumt}{\underset{(d,2)=1}{\sum \nolimits^{*}} \widetilde w\left( \frac dX \right) }

\theoremstyle{plain}
\newtheorem{conj}{Conjecture}
\newtheorem{remark}[subsection]{Remark}

\makeatletter
\def\widebreve{\mathpalette\wide@breve}
\def\wide@breve#1#2{\sbox\z@{$#1#2$}%
     \mathop{\vbox{\m@th\ialign{##\crcr
\kern0.08em\brevefill#1{0.8\wd\z@}\crcr\noalign{\nointerlineskip}%
                    $\hss#1#2\hss$\crcr}}}\limits}
\def\brevefill#1#2{$\m@th\sbox\tw@{$#1($}%
  \hss\resizebox{#2}{\wd\tw@}{\rotatebox[origin=c]{90}{\upshape(}}\hss$}
\makeatletter

\title[Sharp lower bounds for moments of quadratic Dirichlet $L$-functions]{Sharp lower bounds for moments of quadratic Dirichlet $L$-functions}

\author{Peng Gao}
\address{School of Mathematical Sciences, Beihang University, Beijing 100191, P. R. China}
\email{penggao@buaa.edu.cn}
\begin{abstract}
 We establish sharp lower bounds for the $k$-th moment in the range $0 \leq k \leq 1$ of the family of quadratic Dirichlet $L$-functions at the central point.
\end{abstract}

\maketitle

\noindent {\bf Mathematics Subject Classification (2010)}: 11M06  \newline

\noindent {\bf Keywords}: moments, quadratic Dirichlet $L$-functions, lower bounds

\section{Introduction}
\label{sec 1}

  As moments of families of $L$-functions at the central point can be applied to address important issues such as non-vanishing results concerning these central values, they have been studied intensively in the literature.  For the family of quadratic Dirichlet $L$-functions,  asymptotic expressions are available only for the first four moments presently. These results are obtained by  M. Jutila \cite{Jutila} for the first two moments, by K. Soundararajan \cite{sound1} for the third moment and Q. Shen \cite{Shen} for the fourth moment. See also \cites{ViTa, DoHo, Young1, Young2, sound1, Sono} for various improvements on the error terms. Besides the above, conjectured asymptotic formulas for various families of $L$-functions are made in the work of J. P. Keating and N. C. Snaith \cite{Keating-Snaith02}, building on the relation with random matrix theory. More precise formulas including lower order terms are further conjectured by J. B. Conrey, D. W. Farmer, J. P. Keating, M. O. Rubinstein and N. C. Snaith in \cite{CFKRS}.

   Much progress has been made towards establishing bounds for moments of $L$-functions of the order of magnitude in agreement with the above conjectures. There are now several general approaches that allow one to achieve this. For upper bounds, one can apply a method due to K. Soundararajan in \cite{Sound01} together with its sharpened version by A. J. Harper \cite{Harper} or a principle built by M. Radziwi{\l\l} and K. Soundararajan in \cite{Radziwill&Sound}. For lower bounds, one can make use of a simple and powerful method developed by Z. Rudnick and K. Soundararajan in \cite{R&Sound, R&Sound1}, or a principle enunciated by W. Heap and K. Soundararajan in \cite{H&Sound} which can be regarded as dual to the corresponding one of Radziwi{\l\l} and Soundararajan concerning upper bounds.

   We now return to the case of quadratic Dirichlet $L$-functions. More specifically, we consider the family $\{ L(s, \chi_{8d}) \}$ with  $\chi_{8d} = \left(\frac{8d}{\cdot} \right)$ being the Kronecker symbol such that $d$ is odd and square-free. For this family, it is conjectured by J. C. Andrade and J. P. Keating \cite{Andrade-Keating01} that for all positive real $k$,
\begin{align*}
  \sumstar_{\substack{ 0<d<X \\ (d,2)=1}}|L(\tfrac{1}{2},\chi_{8d})|^k \sim C_kX(\log X)^{\frac{k(k+1)}{2}},
\end{align*}
   where $\sumstar$ denotes the sum over square-free integers and $C_k$ are explicit constants.  Owing to the work of \cite{Harper, Sound01, R&Sound, R&Sound1, Radziwill&Sound1}, lower and upper bounds of the conjectured order of magnitude for the above family have been established for all real $k \geq 1$ unconditionally and for all real $k \geq 0$ under the generalized Riemann hypothesis, respectively.

   In \cite{Gao2021-2}, the author applied the above mentioned upper bounds principle of Radziwi{\l\l} and Soundararajan to obtain sharp upper bounds for the $k$-th moment of the family $\{ L(\half, \chi_{8d}) \}$ unconditionally for $0 \leq k \leq 2$. In this paper, based on the lower bounds principle of Heap and Soundararajan in \cite{H&Sound} together with an idea in \cite{Gao2021-2}, we further explore the commensurate lower bounds for the same family. As we have indicated above, sharp bounds for the $k$-th moment are already known for any real $k \geq 1$. Thus, we focus on the case $0 \leq k \leq 1$ here, even though our approach in this paper can be extended to treat the case for $k \geq 1$ as well.  Our result is as follows.
\begin{theorem}
\label{thmlowerbound}
   For any real number $k$ such that  $0 \leq k \leq 1$, we have
\begin{align*}
   \sumstar_{\substack{ 0<d<X \\ (d,2)=1}}|L(\tfrac{1}{2},\chi_{8d})|^{k} \gg_k X(\log X)^{\frac{k(k+1)}{2}}.
\end{align*}
\end{theorem}

\section{Plan of Proof}
\label{sec 2'}

  We may assume that $X$ is a large number and we denote $\Phi$ for a smooth, non-negative function compactly supported on $[1/8,7/8]$ such that $\Phi(x) \leq 1$ for all $x$ and $\Phi(x) =1$ for $x\in [1/4,3/4]$. For convenience, we replace $k$ by $2k$ and assume that $0<k <1/2$ in the rest of the paper. Upon dividing $0<d<X$ into dyadic blocks, we see that in order to prove Theorem \ref{thmlowerbound}, it suffices to show that
\begin{align}
\label{lowerbound}
   \sumstar_{\substack{ 0<d<X \\ (d,2)=1}}|L(\tfrac{1}{2},\chi_{8d})|^{2k}\Phi(\frac dX) \gg_k X(\log X)^{\frac{2k(2k+1)}{2}}.
\end{align}
  To achieve this, we recall the lower bounds principle of Heap and Soundararajan given in \cite{H&Sound} for our setting. This builds on the observation that a result of B. Hough \cite{Hough} on the distribution of $\log |L(\half, \chi_{8d})|$ shows that the behaviour on average of  $|L(\tfrac{1}{2},\chi_{8d})|$ can be modeled by the quantity $(\log d)^{1/2}\exp({\mathcal P}(d))$, where
\begin{align*}
{\mathcal P}(d) = \sum_{p \leq z} \frac{1}{\sqrt{p}} \chi_{8d}(p),
\end{align*}
  with $z=X^{1/(\log \log X)^2}$. Now, the presence of $\Phi$ restricts the size of $\log d$ to be about $\log X$, so we may ignore it to see that the expression
\begin{align}
\label{prodLP}
 \sumstar_{\substack{ 0<d<X \\ (d,2)=1}}|L(\tfrac{1}{2},\chi_{8d})|^{k_1}\exp(k_2{\mathcal P}(d))
\end{align}
  should provide enough information towards our understanding of $|L(\tfrac{1}{2},\chi_{8d})|^{2k}$ on average as long as $k_1+k_2=2k$. As expanding $\exp(k_2{\mathcal P}(d))$ into a Dirichlet polynomial would result in an expression too long to control, the next idea is to approximate $\exp(k_2{\mathcal P}(d))$ by a suitable Taylor expansion (upon realizing that ${\mathcal P}(d)$ is often small in size). Further noticing that larger values of primes should contribute significantly less to the value of ${\mathcal P}(d)$ so that instead of approximating $\exp(k_2{\mathcal P}(d))$ by a single Taylor expansion, it is desirable to split ${\mathcal P}(d)$ into sums over primes of different ranges and approximate each corresponding exponential by tailoring a suitable Taylor polynomial.

 For this purpose,  we define for any non-negative integer $\ell$ and any real number $x$,
\begin{equation*}
E_{\ell}(x) = \sum_{j=0}^{\ell} \frac{x^{j}}{j!}.
\end{equation*}
  We also let $N, M$ be two large natural numbers (depending on $k$ only) and denote $\{ \ell_j \}_{1 \leq j \leq R}$ for a sequence of even natural numbers such that $\ell_1= 2\lceil N \log \log X\rceil$ and $\ell_{j+1} = 2 \lceil N \log \ell_j \rceil$ for $j \geq 1$, where $R$ is defined to the largest natural number satisfying $\ell_R >10^M$.  We may assume that $M$ is so chosen so that we have $\ell_{j} > \ell_{j+1}^2$ for all $1 \leq j \leq R-1$ and this further implies that we have
\begin{align}
\label{sumoverell}
  R \ll \log \log \ell_1, \quad  \sum^R_{j=1}\frac 1{\ell_j} \leq \frac 2{\ell_R}.
\end{align}
  We denote ${ P}_1$ for the set of odd primes not exceeding $X^{1/\ell_1^2}$ and
${ P_j}$ for the set of primes lying in the interval $(X^{1/\ell_{j-1}^2}, X^{1/\ell_j^2}]$ for $2\le j\le R$.

  We set
\begin{equation*}
{\mathcal P}_j(d) = \sum_{p\in P_j} \frac{1}{\sqrt{p}} \chi_{8d}(p),
\end{equation*}
 and for any real number $\alpha$,
\begin{align}
\label{defN}
{\mathcal N}_j(d, \alpha) = E_{\ell_j} (\alpha {\mathcal P}_j(d)), \quad \mathcal{N}(d, \alpha) = \prod_{j=1}^{R} {\mathcal N}_j(d,\alpha).
\end{align}
  Here we notice that it follows from \cite[Lemma 1]{Radziwill&Sound} that the quantities defined in \eqref{defN} are all positive. Instead of examining \eqref{prodLP}, we may now look at
\begin{align}
\label{prodLN}
 \sumstar_{\substack{ 0<d<X \\ (d,2)=1}}|L(\tfrac{1}{2},\chi_{8d})|^{k_1}\prod_{2 \leq j \leq J}\mathcal{N}(d, k_j),
\end{align}
 for real numbers $k_i$ satisfying $\displaystyle \sum_{1 \leq j \leq J}k_j=2k$.  It remains to apply H\"older's inequality to bound the expression in \eqref{prodLN} from above by averages of various quantities that are of the same form as the one appearing in \eqref{prodLN}, one of which being $|L(\tfrac{1}{2},\chi_{8d})|^{2k}$. Thus, if one is able to bound those other quantities on average from above and the expression in \eqref{prodLN} from below, then one should be able to obtain a lower bound estimation for the $2k$-th moment.

  More concretely, we may adapt the choices for $k_i$ in \eqref{prodLN} as those used in \cite{H&Sound} to see that, via H\"older's inequality,
\begin{align*}
 & \sumstar_{(d,2)=1}|L(\half, \chi_{8d})| \mathcal{N}(d, k) \mathcal{N}(d, k-1)   \Phi(\frac dX) \\
 \leq & \Big ( \sumstar_{(d,2)=1}|L(\half, \chi_{8d})|^{2k} \Phi(\frac dX)\Big )^{1/2}\Big ( \sumstar_{(d,2)=1}|L(\half, \chi_{8d})\mathcal{N}(d, k-1)|^2 \Phi(\frac dX) \Big)^{(1-k)/2}\Big ( \sumstar_{(d,2)=1}  \mathcal{N}(d, k)^{2/k}\mathcal{N}(d, k-1)^{2} \Phi(\frac dX) \Big)^{k/2}.
\end{align*}
  Although the above may allow one to obtain lower bounds estimation for the $2k$-th moment, it nevertheless requires estimations on averages of $|L(\half, \chi_{8d})| \mathcal{N}(d, k) \mathcal{N}(d, k-1)$ and $|L(\half, \chi_{8d})\mathcal{N}(d, k-1)|^2$, which we find a little inconvenient. Thus, we incorporate an idea given in \cite{Gao2021-2} by observing that it follows from \cite[Lemma 4.1]{Gao2021-2} that we have
\begin{align*}
 \mathcal{N}(d, \alpha)\mathcal{N}(d, -\alpha)  \geq 1.
\end{align*}
  We deduce from this that for any real $c$ such that $0<c<1$,
\begin{align*}
\begin{split}
 & \sumstar_{(d,2)=1}L(\half, \chi_{8d}) \mathcal{N}(d,2k-1)  \Phi(\frac dX)  \leq \sumstar_{(d,2)=1}|L(\half, \chi_{8d})| \mathcal{N}(d,2k-1)  \Phi(\frac dX) \\
 \leq & \sumstar_{(d,2)=1}|L(\half, \chi_{8d})|^{c}\cdot |L(\half, \chi_{8d})|^{1-c} \mathcal{N}(2k-2, d)^{(1-c)/2}  \cdot \mathcal{N}(d,2k-1) \mathcal{N}(d, 2-2k)^{(1-c)/2} \Phi(\frac dX).
\end{split}
\end{align*}

  Applying H\"older's inequality with exponents being $2k/c, 2/(1-c), ((1+c)/2-c/(2k))^{-1}$ to the last sum above, we deduce that
\begin{align}
\label{basicbound1}
\begin{split}
 & \sumstar_{(d,2)=1}L(\half, \chi_{8d}) \mathcal{N}(d, 2k-1)   \Phi(\frac dX) \\
 \leq & \Big ( \sumstar_{(d,2)=1}|L(\half, \chi_{8d})|^{2k} \Phi(\frac dX)\Big )^{c/(2k)}\Big ( \sumstar_{(d,2)=1}|L(\half, \chi_{8d})|^2 \mathcal{N}(d, 2k-2)\Phi(\frac dX) \Big)^{(1-c)/2} \\
 & \times \Big ( \sumstar_{(d,2)=1}  \mathcal{N}(d, 2k-1)^{((1+c)/2-c/(2k))^{-1}}\mathcal{N}(d, 2-2k)^{(1-c)/2 \cdot ((1+c)/2-c/(2k))^{-1}} \Phi(\frac dX) \Big)^{(1+c)/2-c/(2k)}.
\end{split}
\end{align}

  We now set
$$(1-c)/2 \cdot ((1+c)/2-c/(2k))^{-1}=2$$
   which implies that $c=(\frac 2k-3)^{-1}$ and that
$$((1+c)/2-c/(2k))^{-1}=\frac {2(2-3k)}{1-2k}.$$
   One checks that the above value of $c$ does satisfy that $0<c<1$ when $0<k<1/2$.  We then deduce from \eqref{basicbound1} that in order to establish a lower bound estimation for the $2k$-th moment, it suffices to establish the following three propositions.
\begin{proposition}
\label{Prop4} With notations as above, we have
\begin{align*}
\sumstar_{(d,2)=1} L(\frac 12,\chi_{8d}){\mathcal N}(d, 2k-1)\Phi\Big(\frac{d}{X}\Big) \gg X (\log X)^{ \frac {(2k)^2+1}{2}}.
\end{align*}
\end{proposition}

\begin{proposition}
\label{Prop5} With notations as above, we have
\begin{align*}
\sumstar_{(d,2)=1}\mathcal{N}(d, 2k-1)^{\frac {2(2-3k)}{1-2k}}\mathcal{N}(d, 2-2k)^{2} \Phi\Big(\frac{d}{X}\Big) \ll X ( \log X  )^{\frac {(2k)^2}{2}}.
\end{align*}
\end{proposition}

\begin{proposition}
\label{Prop6} With notations as above, we have
\begin{align*}
\sumstar_{(d,2)=1} L(\frac 12,\chi_{8d})^2{\mathcal N}(d, 2k-2)\Phi\Big(\frac{d}{X}\Big)  \ll X ( \log X  )^{\frac {(2k)^2+2}{2}}.
\end{align*}
\end{proposition}

   In fact, combining \eqref{basicbound1} with the above three propositions, we see that
\begin{align*}
\begin{split}
 & X (\log X)^{ \frac {(2k)^2+1}{2}} \ll \Big ( \sumstar_{(d,2)=1}|L(\half, \chi_{8d})|^{2k} \Phi(\frac dX)\Big )^{c/(2k)}\Big ( X (\log X)^{ \frac {(2k)^2+2}{2}}\Big)^{(1-c)/2}\Big (  X (\log X)^{ \frac {(2k)^2}{2}}\Big)^{(1+c)/2-c/(2k)}.
\end{split}
\end{align*}
  The desired lower bound given in \eqref{lowerbound} follows immediately from this.

  Notice that Proposition \ref{Prop6} can be established similar to \cite[Proposition 3.1]{Gao2021-2} so that it remains to prove Propositions \ref{Prop4} and \ref{Prop5} in the rest of the paper.

\section{Preliminaries}
\label{sec 2}

  We include in this section a few lemmas that are needed in our proofs. In what follows, we denote the letter $p$ for a prime number and the symbol $\square$ for a perfect square. We define $\delta_{n=\square}$ to be $1$ when $n=\square$ and $0$ otherwise. For the function $\Phi$ described in Section \ref{sec 2'},  we associate its Mellin transform ${\widehat \Phi}(s)$ for any complex number $s$ by
\begin{equation*}
{\widehat \Phi}(s) = \int_{0}^{\infty} \Phi(x)x^{s}\frac {dx}{x}.
\end{equation*}

  We first recall the following two lemmas that are given in \cite{Gao2021-2} as Lemma 2.2 and 2.3 there, respectively.
\begin{lemma}
\label{RS} Let $x \geq 2$. We have, for some constant $b$,
$$
\sum_{p\le x} \frac{1}{p} = \log \log x + b+ O\Big(\frac{1}{\log x}\Big).
$$
 Also, for any integer $j \geq 1$, we have
$$
\sum_{p\le x} \frac {(\log p)^j}{p} = \frac {(\log x)^j}{j} + O((\log x)^{j-1}).
$$
\end{lemma}

\begin{lemma} \label{PropDirpoly}  For large $X$ and any odd positive integer $n$, we have
\begin{align}
\label{meancharsum}
\sumstar_{\substack{(d,2)=1}} \chi_{8d}(n) \Phi\Big(\frac{d}{X}\Big)=
\displaystyle \delta_{n=\square}{\widehat \Phi}(1) \frac{2X}{3\zeta(2)} \prod_{p|n} \Big(\frac p{p+1}\Big) + O(X^{\frac 12+\epsilon} \sqrt{n} ).
\end{align}
\end{lemma}

  Next, similar to \cite[Lemma 2.4]{Gao2021-2}, by setting $Y=X^{1/4}, M=1$ in \cite[Proposition 1.1-1.2]{sound1}, we have the following asymptotic formula concerning the twisted first moment of quadratic Dirichlet $L$-functions.
\begin{lemma}
\label{Prop1}
  With notations as above and writing any odd $l$ as  $l=l_1l^2_2$ with $l_1$ square-free, we have for any $\varepsilon>0$,
\begin{align*}
\begin{split}
 \sumstar_{(d,2)=1}L(\half, \chi_{8d})\chi_{8d}(l)\Phi(\frac dX)
=&  \frac{C \widehat{\Phi}(1)}{\zeta(2)}\frac {1}{\sqrt{l_1}g(l)}X \Big (\log \frac {\sqrt{X}}{l_1}+C_2+\sum_{\substack{p | l}} \frac {C_2(p)}{p}\log p \Big )+O \left(X^{\frac 34+\varepsilon}l^{\half+\varepsilon}\right ),
\end{split}
\end{align*}
  where  $C=\frac 13 \displaystyle \prod_{\substack{p \geq 3}}\left (1-\frac 1{p(p+1)} \right )$ and $g(l)=\displaystyle \prod_{p | l}\Big ( \frac {p+1}{p} \Big )\left (1-\frac 1{p(p+1)} \right ).$
  Also, $C_2$ is a constant depending only on $\Phi$ and $C_2(p) \ll 1$ for all $p$.
\end{lemma}

  Lastly, we present a result which is analogue to \cite[Lemma 1]{H&Sound} and is needed in the proof of Proposition \ref{Prop5}.
\begin{lemma}
\label{lem1}
 For $1 \le j\le R$, we have
$$
{\mathcal N}_j(d, 2k-1)^{\frac {2(2-3k)}{1-2k}} {\mathcal N}_j(d, 2-2k)^{2} \le
{\mathcal N}_j(d, 2k) \Big(1+ O\big(e^{-\ell_j} \big ) \Big) + {\mathcal Q}_j(d),
$$
where the implied constants are absolute, and
$$
{\mathcal Q}_j(d) =\Big( \frac{24 {\mathcal P}_j(d)}{\ell_j} \Big)^{2r_k\ell_j},
$$
  with $r_k=1+(2-3k)/(1-2k)$.
\end{lemma}
\begin{proof}  We note first that when $|z| \le aK/20$ for any constant $0<a \leq 2$, we have
$$
\Big| \sum_{r=0}^K \frac{z^r}{r!} - e^z \Big| \le \frac{|az|^{K}}{K!} \le \Big(\frac{a e}{20}\Big)^{K}.
$$
  We deduce from this that when $|{\mathcal P}_j(d)| \le \ell_j/20$,
\begin{align*}
{\mathcal N}_j(d, 2k-1)=& \exp( (2k-1) {\mathcal P}_j(d))\Big( 1+  O\Big(\exp( (2k-1) |{\mathcal P}_j(d)|)\Big(\frac{(1-2k) e}{20}\Big)^{\ell_j} \Big ) \\
= & \exp( (2k-1) {\mathcal P}_j(d))\Big( 1+  O\Big((1-2k)e^{-\ell_j} \Big )\Big ).
\end{align*}
  Similarly, we have
\begin{align*}
{\mathcal N}_j(d, 2-2k)= & \exp( (2-2k) {\mathcal P}_j(d))\Big( 1+  O\Big(e^{-\ell_j} \Big ) \Big ).
\end{align*}
 The above estimations then allow us to see that when $|{\mathcal P}_j(d)| \le \ell_j/20$,
\begin{align}
\label{est1}
{\mathcal N}_j(d, 2k-1)^{\frac {2(2-3k)}{1-2k}} {\mathcal N}_j(d, 2-2k)^{2}
&= \exp( 2k {\mathcal P}_j(d))\Big( 1+ O\big( e^{-\ell_j} \big) \Big) = {\mathcal N}_j(d, 2k) \Big( 1+ O\big(e^{-\ell_j} \big) \Big).
\end{align}

 Next, notice that when $|{\mathcal P}_j(d)| \ge \ell_j/20$, we have
\begin{align}
\label{4.2}
\begin{split}
|{\mathcal N}_j(d, 2-2k)| &\le \sum_{r=0}^{\ell_j} \frac{|2{\mathcal P}_j(d)|^r}{r!} \le
|{\mathcal P}_j(d)|^{\ell_j} \sum_{r=0}^{\ell_j} \Big( \frac{20}{\ell_j}\Big)^{\ell_j-r} \frac{2^r}{r!}   \le \Big( \frac{24 |{\mathcal P}_j(d)|}{\ell_j}\Big)^{\ell_j} .
\end{split}
\end{align}
  The assertion of the lemma now follows from \eqref{est1}, \eqref{4.2} together with the observation that the same bound above holds for $|{\mathcal N}_j(d, 2k-1)|$ as well.
\end{proof}

\section{Proof of Proposition \ref{Prop4}}
\label{sec 4}

   Denote $\Omega(n)$ for the number of distinct prime powers dividing $n$ and $w(n)$ for the multiplicative function such that $w(p^{\alpha}) = \alpha!$ for prime powers $p^{\alpha}$.  Let $b_j(n), 1 \leq j \leq R$ be functions such that $b_j(n)=1$ when $n$ is composed of at most $\ell_j$ primes, all from the interval $P_j$. Otherwise, we define $b_j(n)=0$. We use these notations to see that
\begin{equation}
\label{5.1}
{\mathcal N}_j(d, 2k-1) = \sum_{n_j} \frac{1}{\sqrt{n_j}} \frac{(2k-1)^{\Omega(n_j)}}{w(n_j)}  b_j(n_j) \chi_{8d}(n_j), \quad 1\le j\le R.
\end{equation}
    Note that each ${\mathcal N}_j(d, 2k-1)$ is a short Dirichlet polynomial since $b_j(n_j)=0$ unless $n_j \leq (X^{1/\ell_j^2})^{\ell_j}=X^{1/\ell_j}$. It follows from this that  ${\mathcal N}(d, 2k-1)$ is also a short Dirichlet polynomial of length at most $X^{1/\ell_1+ \ldots +1/\ell_R} < X^{2/10^{M}}$ by \eqref{sumoverell}.

 We use \eqref{5.1} to expand the term ${\mathcal N}(d, 2k-1)$ and apply Lemma  \ref{PropDirpoly} to evaluate it. As ${\mathcal N}(d, 2k-1)$ is a short Dirichlet polynomial, we may ignore the error term in Lemma \ref{PropDirpoly} to consider only the main term contribution. Upon writing $n_j=(n_j)_1(n_j)_{2}^2$ with $(n_j)_{1}$ being square-free, we see that
\begin{align*}
 & \sumstar_{(d,2)=1} L(\frac 12,\chi_{8d}){\mathcal N}(d, 2k-1)\Phi\Big(\frac{d}{X}\Big) \\
 \gg & X   \sum_{n_1, \cdots, n_{R} }  \Big( \prod_{j=1}^{R} \frac{1}{\sqrt{n_j(n_j)_{1}}}
\frac{(2k-1)^{\Omega(n_j)}}{w(n_j) } b_j(n_j) \frac {1}{g(n_{j})}  \Big) \Big (\log \Big ( \frac {\sqrt{X}}{(n_1)_1 \cdots (n_{R})_1} \Big )+C_2+\sum_{\substack{p | n_1 \cdots n_{R} }} \frac {C_2(p)}{p}\log p \Big ).
\end{align*}
  We may further ignore the contribution from the terms involving $C_2+\sum_{\substack{p | n_1 \cdots n_{R} }} \frac {C_2(p)}{p}\log p$ as one may show using our arguments in the paper that the contribution is $\ll  X (\log X)^{ \frac {(2k)^2+1}{2}-1}$. Thus we deduce that
\begin{align*}
  \sumstar_{(d,2)=1} L(\frac 12,\chi_{8d}){\mathcal N}(d, 2k-1)\Phi\Big(\frac{d}{X}\Big)
\gg & X   \sum_{n_1, \cdots, n_{R} }  \Big( \prod_{j=1}^{R} \frac{1}{\sqrt{n_j(n_j)_{1}}}
\frac{(2k-1)^{\Omega(n_j)}}{w(n_j) } b_j(n_j) \frac {1}{g(n_{j})}  \Big) \Big (\log \Big ( \frac {\sqrt{X}}{(n_1)_1 \cdots (n_{R})_1} \Big )\Big ) \\
=& S_1-S_2,
\end{align*}
   where
\begin{align*}
 S_1= & \frac 12 X \log X   \sum_{n_1, \cdots, n_{R} }  \Big( \prod_{j=1}^{R} \frac{1}{\sqrt{n_j(n_j)_{1}}}
\frac{(2k-1)^{\Omega(n_j)}}{w(n_j) } b_j(n_j) \frac {1}{g(n_{j})}  \Big),  \\
 S_2= &  X   \sum_{n_1, \cdots, n_{R} }  \Big( \prod_{j=1}^{R} \frac{1}{\sqrt{n_j(n_j)_{1}}}
\frac{(2k-1)^{\Omega(n_j)}}{w(n_j) } b_j(n_j) \frac {1}{g(n_{j})}  \Big) \log \big ( \prod^R_{i=1} (n_i)_1  \big ).
\end{align*}
   It remains to bound $S_1$ from below and $S_2$ from above. We bound $S_1$ first by recasting it as
\begin{align*}
 S_1= & \frac 12 X \log X   \prod_{j=1}^{R} \sum_{n_j }  \Big (\frac{1}{\sqrt{n_j(n_j)_{1}}}
\frac{(2k-1)^{\Omega(n_j)}}{w(n_j) } b_j(n_j) \frac {1}{g(n_{j})}  \Big).
\end{align*}
  We consider the sum above over $n_j$ for a fixed $1 \leq j \leq R$. Note that the factor $b_j(n_j)$ restricts $n_j$ to have all prime factors in $P_j$ such that $\Omega(n_j) \leq \ell_j$. If we remove the restriction on $\Omega(n_j)$, then the sum becomes
\begin{align}
\label{6.02}
\begin{split}
& \prod_{\substack{p\in P_j }} \Big( \sum_{i=0}^{\infty} \frac{1}{p^i} \frac{(2k-1)^{2i}}{(2i)!g(p^{2i})} + \sum_{i=0}^{\infty} \frac{1}{p^{i+1}}
\frac{(2k-1)^{2i+1}}{(2i+1)!}\frac {1}{g(p^{2i+1})}\Big)
= \prod_{\substack{p\in P_j }}\Big (1- \big(\frac {(2k-1)^{2}}{2}+2k-1 \big )\frac 1p \Big )^{-1} D(p),
\end{split}
\end{align}
  where
\begin{align*}
\begin{split}
 D(p)=& \Big (1- \big(\frac {(2k-1)^{2}}{2}+2k-1 \big )\frac 1p \Big )\Big( \sum_{i=0}^{\infty} \frac{1}{p^i} \frac{(2k-1)^{2i}}{(2i)!g(p^{2i})} + \sum_{i=0}^{\infty} \frac{1}{p^{i+1}}
\frac{(2k-1)^{2i+1}}{(2i+1)!}\frac {1}{g(p^{2i+1})}\Big).
\end{split}
\end{align*}
  Note that for $i \geq 2$, we have
\begin{align*}
\begin{split}
  \frac{1}{p^i} \frac{(2k-1)^{2i}}{(2i)!g(p^{2i})} + \frac{1}{p^{i+1}}
\frac{(2k-1)^{2i+1}}{(2i+1)!}\frac {1}{g(p^{2i+1})} \geq 0.
\end{split}
\end{align*}
  Observe further the estimation that
\begin{align}
\label{gest}
\begin{split}
   1-\frac 1p \leq \frac 1{g(p)} \leq 1.
\end{split}
\end{align}
  We then deduce that
\begin{align*}
\begin{split}
 D(p) \geq & \Big (1- \big(\frac {(2k-1)^{2}}{2}+2k-1 \big )\frac 1p \Big )\Big (1+\big(\frac {(2k-1)^{2}}{2}+2k-1 \big )\frac 1{pg(p)}+\frac {(2k-1)^3}{6p^2g(p)} \Big ) \\
 \geq & \Big (1- \big(\frac {(2k-1)^{2}}{2}+2k-1 \big )\frac 1p \Big )\Big (1+\big(\frac {(2k-1)^{2}}{2}+2k-1 \big )\frac 1{p}+\frac {(2k-1)^3}{6p^2} \Big ) \\
 \geq & 1-\big(\frac {(2k-1)^{2}}{2}+2k-1 \big )^2\frac 1{p^2} - \Big (1- \big(\frac {(2k-1)^{2}}{2}+2k-1 \big )\frac 1p \Big )\frac {(1-2k)^3}{6p^2}  \\
 \geq & 1-\frac 1{4p^2}-\frac 5{24p^2},
\end{split}
\end{align*}
  where the last estimation above follows by noting that
$$1- \big(\frac {(2k-1)^{2}}{2}+2k-1 \big )\frac 1p = 1+\frac {1-4k^2}{2p} \leq 1+\frac 1{4}.$$

   We further note that we have $(2k-1)^2/2+2k-1<0$ when $0<k<1/2$ and that $-\log (1+x)>-x$ for all $x>0$. This implies that
\begin{align*}
\begin{split}
&  \prod_{\substack{p\in P_j }}\Big (1- \big(\frac {(2k-1)^{2}}{2}+2k-1 \big )\frac 1p \Big )^{-1}
\geq  \exp \Big ( \big(\frac {(2k-1)^{2}}{2}+2 k-1 \big )\sum_{\substack{p\in P_j}}\frac 1p \Big ).
\end{split}
\end{align*}

   We then deduce from the above estimations that the left side expression in \eqref{6.02} is
\begin{align}
\label{firstlowerbound}
\begin{split}
\geq & \exp \Big ( \big(\frac {(2k-1)^{2}}{2}+2 k-1 \big )\sum_{\substack{p\in P_j}}\frac 1p \Big )\prod_{\substack{p\in P_j }} (1-\frac 1{p^2}).
\end{split}
\end{align}

  On the other hand, using Rankin's trick by noticing that $2^{\Omega(n_1)-\ell_1}\ge 1$ if $\Omega(n_1) > \ell_1$,  we see that the error introduced this way does not exceed
\begin{align*}
\begin{split}
 & \sum_{n_j} \frac{1}{\sqrt{n_j(n_{j})_{1}}}
\frac{|2k-1|^{\Omega(n_j)}}{w(n_j) } 2^{\Omega(n_j)-\ell_j} \frac {1}{g(n_{j})} \\
\le & 2^{-\ell_j} \prod_{\substack{p\in P_j }} \Big( 1+\sum_{i=1}^{\infty} \frac{1}{p^i} \frac{(2k-1)^{2i}2^{2i}}{(2i)!} \frac {1}{g(p^{2i})} + \sum_{i=0}^{\infty} \frac{1}{p^{i+1}}
\frac{|2k-1|^{2i+1}2^{2i+1}}{(2i+1)!}\frac {1}{g(p^{2i+1})}\Big) \\
\le & 2^{-\ell_j} \prod_{\substack{p\in P_j }} \Big( 1+\sum_{i=1}^{\infty} \frac{1}{p^i} \frac{(2k-1)^{2i}2^{2i}}{(2i)!}  + \sum_{i=0}^{\infty} \frac{1}{p^{i+1}}
\frac{|2k-1|^{2i+1}2^{2i+1}}{(2i+1)!}\Big) \\
\leq & 2^{-\ell_j}
\exp\Big( \big (2(2k-1)^2+2(1-2k)\big ) \sum_{p\in P_j} \frac{1}{p} +\sum_{p\in P_j} R(p) \Big ),
\end{split}
\end{align*}
 where
\begin{align*}
\begin{split}
 R(p) =& \sum_{i=2}^{\infty} \frac{1}{p^i} \frac{(2k-1)^{2i}2^{2i}}{(2i)!}+\sum_{i=1}^{\infty} \frac{1}{p^{i+1}}
\frac{|2k-1|^{2i+1}2^{2i+1}}{(2i+1)!} \\
\leq & \frac 1{p^2} \sum_{i=2}^{\infty}  \frac{2^{2i}}{(2i)!}+\frac 2{p^2}\sum_{i=1}^{\infty}
\frac{2^{2i}}{(2i+1)!} \\
\leq & \frac 3{p^2}\sum_{i=1}^{\infty}  \frac{2^{2i}}{(2i)!} \leq  \frac 3{p^2}\sum_{i=1}^{\infty}  \frac{2^{i}}{i!} \leq \frac {3e^2}{p^2}.
\end{split}
\end{align*}
  It follows that the error is hence
\begin{align}
\label{errorbound0}
\begin{split}
 \leq & 2^{-\ell_j}
\exp\Big( \big (2(2k-1)^2+2(1-2k)\big ) \sum_{p\in P_j} \frac{1}{p} +\sum_{n \geq 1} \frac {3e^2}{n^2} \Big ) \\
\leq & 2^{-\ell_j}\exp(\frac {(e\pi)^2}{2}) \exp\Big( \big (\frac 32(2k-1)^2+3(1-2k)\big ) \sum_{p\in P_j} \frac{1}{p} \Big ) \times \exp \Big ( \big(\frac {(2k-1)^{2}}{2}+2 k-1 \big )\sum_{p\in P_j} \frac{1}{p} \Big ).
\end{split}
\end{align}

  We notice that the expression given in \eqref{firstlowerbound} satisfies
\begin{align}
\label{lowerbound1}
\begin{split}
\geq & \exp \Big ( \big(\frac {(2k-1)^{2}}{2}+2 k-1 \big )\sum_{\substack{p\in P_j}}\frac 1p \Big )\prod_{\substack{p  }} (1-\frac 1{p^2})=\frac 6{\pi^2}\exp \Big ( \big(\frac {(2k-1)^{2}}{2}+2 k-1 \big )\sum_{\substack{p\in P_j}}\frac 1p \Big ) .
\end{split}
\end{align}
  We may take $N$ large enough so that by Lemma \ref{RS}, we have that for all $1 \leq j \leq R$,
\begin{align}
\label{boundsforsumoverp}
  \frac 1{2N} \ell_j \leq \sum_{p \in P_j}\frac 1{p} \leq \frac 2N \ell_j.
\end{align}
 We may also take $M$ large enough to ensure that every $\ell_j, 1 \leq j \leq R$ is large. We then deduce from \eqref{errorbound0} and \eqref{lowerbound1} that the error introduced above is
\begin{align}
\label{lowerbound2}
\begin{split}
\leq & 2^{-\ell_j/2}\exp \Big ( \big(\frac {(2k-1)^{2}}{2}+2 k-1 \big )\sum_{\substack{p\in P_j}}\frac 1p \Big )\prod_{\substack{p \in P_j }} (1-\frac 1{p^2}).
\end{split}
\end{align}

  Combining \eqref{firstlowerbound} and \eqref{lowerbound2}, we see that the  sum over $n_j$ for each $j$, $1 \leq j \leq R$ in the expression of $S_1$ is
\begin{align*}
\begin{split}
\geq & (1-2^{-\ell_j/2})\exp \Big ( \big(\frac {(2k-1)^{2}}{2}+2 k-1 \big )\sum_{\substack{p\in P_j}}\frac 1p \Big )\prod_{\substack{p \in P_j }} (1-\frac 1{p^2}).
\end{split}
\end{align*}
  It follows from this that we have
\begin{align}
\label{S1bound}
\begin{split}
 S_1 \geq & \frac 12 X\log X \prod^R_{j=1}\Big ( (1-2^{-\ell_j/2})\exp \Big ( \big(\frac {(2k-1)^{2}}{2}+2 k-1 \big )\sum_{\substack{p\in P_j}}\frac 1p \Big )\prod_{\substack{p \in P_j }} (1-\frac 1{p^2}) \Big ) \\
 \geq & \frac 12 X\log X \prod_{\substack{p}} (1-\frac 1{p^2})\prod^R_{j=1}  \Big ( 1-2^{-\ell_j/2} \Big )\prod^R_{j=1} \exp \Big ( \big(\frac {(2k-1)^{2}}{2}+2 k-1 \big )\sum_{\substack{p\in  P_j}}\frac 1p \Big ) \Big ) \\
 \geq & \frac 1{2\zeta(2)} X\log X\Big ( 1- \sum^R_{j=1}   2^{-\ell_j/2} \Big )\prod^R_{j=1} \exp \Big ( \big(\frac {(2k-1)^{2}}{2}+2 k-1 \big )\sum_{\substack{p\in  P_j}}\frac 1p \Big ) \Big ),
\end{split}
\end{align}
 where the last estimation above follows by noting that $\prod^R_{i=1}(1-x_i) \geq 1-\sum^R_{i=1}x_i$ for positive real numbers $x_i$ satisfying $\sum^R_{i=1}x_i \leq 1$.

  Next, we estimate $S_2$ by writing  $\log \big ( \prod^R_{i=1} (n_i)_1  \big )$ as a sum of logarithms of primes dividing $\prod^R_{i=1} (n_i)_1$ to see that
\begin{align}
\label{6.1}
\begin{split}
 S_2 \leq & X \sum_{q \in \bigcup P_j} \sum_{l \geq 0}  \Big( \frac{\log q}{q^{l+1}}\frac{(1-2k)^{2l+1}}{(2l+1)!} \frac {1}{g(q^{2l+1})}\Big)\prod_{i=1}^{R} \Big( \sum_{(n_i, q)=1  } \frac{1}{\sqrt{n_i (n_{i})_{1}}}
\frac{(2k-1)^{\Omega(n_i)}}{w(n_i) } \widetilde{b}_{i, l}(n_i) \frac {1}{g(n_{i})} \Big),
\end{split}
\end{align}
  where we define $\widetilde{b}_{i, l}(n_i)=b_i(n_iq^{l})$ for the unique index $i$ ($1 \leq i \leq R$) such that $b_i(q) \neq 0$ and $\widetilde{b}_{i, l}(n_i)=b_{i}(n_i)$ otherwise.

  As above, if we remove the restriction of $\widetilde{b}_{i, l}$ on $\Omega(n_i)$, then the sum over $n_i$ becomes
\begin{align}
\label{6.2}
\begin{split}
& \prod_{\substack{p \in P_i \\ (p,q)=1}} \Big( \sum_{m=0}^{\infty} \frac{1}{p^m} \frac{(2k-1)^{2m}}{(2m)!g(p^{2m})} + \sum_{m=0}^{\infty} \frac{1}{p^{m+1}}
\frac{(2k-1)^{2m+1}}{(2m+1)!}\frac {1}{g(p^{2m+1})}\Big).
\end{split}
\end{align}
  Note that for $m \geq 1$, we have
\begin{align*}
\begin{split}
  \frac{1}{p^m} \frac{(2k-1)^{2m}}{(2m)!g(p^{2m})} + \frac{1}{p^{m}}
\frac{(2k-1)^{2m-1}}{(2m-1)!}\frac {1}{g(p^{2m-1})} \leq 0.
\end{split}
\end{align*}
  It follows from this that the expression in \eqref{6.2} is
\begin{align}
\label{6.2'}
\begin{split}
 \leq & \prod_{\substack{p \in P_i \\ (p,q)=1}} \Big (1+\big(\frac {(2k-1)^{2}}{2}+2k-1 \big )\frac 1{pg(p)} \Big ) \leq
 \exp \Big ( \big(\frac {(2k-1)^{2}}{2}+2k-1 \big )\sum_{\substack{p \in P_i \\ (p,q)=1}}\frac 1{pg(p)} \Big ),
\end{split}
\end{align}
 where the last estimation above follows from the  observation that $1 + x \leq e^x$ for any real x.

 Using further the estimation given in \eqref{gest}, we deduce from \eqref{6.2'} that the expression in \eqref{6.2} is
\begin{align*}
\begin{split}
 \leq & \exp \Big ( \big(\frac {(2k-1)^{2}}{2}+2k-1 \big )\sum_{\substack{p \in P_i \\ (p,q)=1}}\frac 1{p}-  \big(\frac {(2k-1)^{2}}{2}+2k-1 \big )\sum_{\substack{p \in P_i }}\frac 1{p^2} \Big ) \\
 \leq & \exp \Big ( \big(\frac {(2k-1)^{2}}{2}+2k-1 \big )\sum_{\substack{p \in P_i \\ (p,q)=1}}\frac 1{p}+\frac 12\sum_{\substack{p \in P_i }}\frac 1{p^2} \Big ).
\end{split}
\end{align*}
  Similar to our discussions above, we see that the error introduced in this process is
\begin{align*}
\begin{split}
\leq & 2^{-\ell_j/2}\exp \Big ( \big(\frac {(2k-1)^{2}}{2}+2k-1 \big )\sum_{\substack{p \in P_i \\ (p,q)=1}}\frac 1{p} \Big ).
\end{split}
\end{align*}

   We deduce from this that
\begin{align*}
\begin{split}
 & \prod_{i=1}^{R} \Big( \sum_{(n_i, q)=1  } \frac{1}{\sqrt{n_i (n_{i})_{1}}}
\frac{(2k-1)^{\Omega(n_i)}}{w(n_i) } \widetilde{b}_{i, l}(n_i) \frac {1}{g(n_{i})} \Big) \\
\leq & \prod_{i=1}^{R} \big (1+2^{-\ell_j/2} \big )\exp \Big ( \big(\frac {(2k-1)^{2}}{2}+2k-1 \big )\sum_{\substack{p \in P_i \\ (p,q)=1}}\frac 1{p}+\frac 12\sum_{\substack{p \in P_i }}\frac 1{p^2} \Big ) \\
\leq & A \times \prod_{i=1}^{R} \exp \Big ( \big(\frac {(2k-1)^{2}}{2}+2k-1 \big )\sum_{\substack{p \in P_i }}\frac 1{p}\Big ),
\end{split}
\end{align*}
  where $A \leq e^{10}$ is a constant. 

  It follows from this and \eqref{6.1} that
\begin{align}
\label{S2bound}
\begin{split}
 S_2 \leq & A X  \prod_{i=1}^{R} \exp \Big ( \big(\frac {(2k-1)^{2}}{2}+2k-1 \big )\sum_{\substack{p \in P_i }}\frac 1{p}\Big ) \times \sum_{q \in \bigcup P_j} \sum_{l \geq 0}  \Big( \frac{\log q}{q^{l+1}}\frac{(1-2k)^{2l+1}}{(2l+1)!} \frac {1}{g(q^{2l+1})}\Big) \\
 \leq & A X \prod_{i=1}^{R} \exp \Big ( \big(\frac {(2k-1)^{2}}{2}+2k-1 \big )\sum_{\substack{p \in P_i }}\frac 1{p}\Big )\times \sum_{q \in \bigcup P_j} \Big (  \frac{\log q}{q} \Big ) \\
 \leq & A X \prod_{i=1}^{R} \exp \Big ( \big(\frac {(2k-1)^{2}}{2}+2k-1 \big )\sum_{\substack{p \in P_i }}\frac 1{p}\Big )\times  \Big (\frac {\log X}{10^{2M}}+O(1) \Big ),
\end{split}
\end{align}
  where the last estimation above follows from Lemma \ref{RS}. 

   Combining \eqref{S1bound} and \eqref{S2bound}, we deduce that, by taking $M$ large enough, 
\begin{align*}
\begin{split}
 S_1-S_2 \gg & X \log X \prod_{i=1}^{R} \exp \Big ( \big(\frac {(2k-1)^{2}}{2}+2k-1 \big )\sum_{\substack{p \in P_i }}\frac 1{p}\Big ).
\end{split}
\end{align*}
  Applying Lemma \ref{RS} again, we see that the proof of the proposition now follows from above.

\section{Proof of Proposition \ref{Prop5}}

   It follows from Lemma \ref{lem1} that we have
\begin{align}
\label{upperboundprodofN}
 \sumstar_{(d,2)=1}  \mathcal{N}(d, 2k-1)^{\frac {2(2-3k)}{1-2k}}\mathcal{N}(d, 2-2k)^{2} \Phi(\frac dX) \le
\sumstar_{(d,2)=1}  \Big ( \prod^R_{j=1} \Big ( {\mathcal N}_j(d, 2k) \Big(1+ O\big(e^{-\ell_j} \big ) \Big) +   {\mathcal Q}_j(d)  \Big )\Big )\Phi(\frac dX).
\end{align}
  We now use the notations in Section \ref{sec 4} to write for $1\le j\le R$,
\begin{align*}
 {\mathcal N}_j(d, 2k)=& \sum_{n_j} \frac{1}{\sqrt{n_j}} \frac{(2k)^{\Omega(n_j)}}{w(n_j)}  b_j(n_j)\chi_{8d}(n_j), \\
  {\mathcal P}_j(d)^{2r_k\ell_j} =&  \sum_{ \substack{ \Omega(n_j) = 2r_k\ell_j \\ p|n_j \implies  p\in P_j}} \frac{1}{\sqrt{n_j}}\frac{(2r_k\ell_j)! }{w(n_j)}\chi_{8d}(n_j),
\end{align*}
  where $r_k=1+(2-3k)/(1-2k)$.

  As
\begin{align}
\label{Stirling}
  (\frac ne)^n \leq n! \leq n(\frac ne)^n,
\end{align}
 we deduce that
\begin{align*}
 & \Big( \frac{24 }{\ell_j} \Big)^{2r_k\ell_j}(2r_k \ell_j)!
\leq  2r_k\ell_j\Big( \frac{48r_k }{e} \Big)^{2r_k\ell_j}.
\end{align*}
  It follows from this that we can write ${\mathcal N}_j(d, 2k) \Big(1+ O\big(e^{-\ell_j} \big ) \Big) + {\mathcal Q}_j(d)$
as a Dirichlet polynomial of the form
\begin{align*}
  D_j(d)=\sum_{n_j \leq X^{2r_k/\ell_j}}\frac{a_{n_j}}{\sqrt{n_j}}\chi_{8d}(n_j)
\end{align*}
   where for some constant $B(k)$ depending on $k$ only,
\begin{align*}
  |a_{n_j}| \leq B(k)^{\ell_j}.
\end{align*}
    We then apply Lemma \ref{PropDirpoly} to evaluate the right side expression in \eqref{upperboundprodofN} above and deduce from estimations in \eqref{sumoverell} that for large $M, N$, the contribution arising from the error term in \eqref{meancharsum} is
\begin{align*}
 \ll B(k)^{\sum^R_{j=1}\ell_j}X^{1/2+\varepsilon}X^{2r_k\sum^R_{j=1}\frac 1{\ell_j}} \ll B(k)^{R\ell_1}X^{1/2+\varepsilon}X^{\frac {4r_k}{\ell_R}} \ll X^{1-\varepsilon}.
\end{align*}

   We may thus focus on the main term contributions, which implies that the right side expression of \eqref{upperboundprodofN} is
\begin{align}
\label{maintermbound}
\begin{split}
 \ll & X \times \sum \Big ( \text{square term in the expansion of } \prod^R_{j=1}D_j(d) \times \text{a corresponding factor } \Big ) \\
=& X \times \prod^R_{j=1} \sum \Big ( \text{square term in } D_j(d) \times \text{a corresponding factor } \Big ),
\end{split}
\end{align}
  where by ``a square term" in a Dirichlet polynomial $\displaystyle \sum_n \frac {a_n}{n^s}$ we mean a term corresponding to $n=\square$ and
\begin{align*}
  \text{a corresponding factor }= \prod_{p | n}\Big ( \frac p{p+1} \Big ).
\end{align*}

   We first note that
\begin{align}
\label{sqinN}
\begin{split}
 &  \sum \text{square term in } {\mathcal N}_j(d, 2k)   \times \text{a corresponding factor }
=     \sum_{n_j=\square} \frac{1}{\sqrt{n_j}} \frac{(2k)^{\Omega(n_j)}}{w(n_j)}  b_j(n_j) \prod_{p | n_j}\Big ( \frac p{p+1} \Big ) \\
\leq & \prod_{p \in P_j}\Big ( 1 + \frac {\big(2k \big )^2}{2p}\Big ( \frac p{p+1} \Big )+\sum_{i \geq 1}\frac {(2k)^{2i}}{p^{2i}(2i)!}\Big ( \frac p{p+1} \Big )\Big ) \leq \prod_{p \in P_j}\Big ( 1 + \frac {\big(2k \big )^2}{2p}\Big ( \frac p{p+1} \Big )+\frac {1}{p^{2}}\Big )\\
\leq & \exp \Big (\frac {\big(2k \big )^2}{2}\sum_{p \in P_j}\frac 1p+\sum_{p \in P_j}\frac {1}{p^{2}}\Big ),
\end{split}
\end{align}
 where we apply the relation that $1+x \leq e^x$ for any real $x$ again to obtain the last estimation above.

  Next, we notice that
\begin{align*}
 &  \sum \text{square term in } {\mathcal Q}_j(d)  \times \text{a corresponding factor } \\
\ll & \Big( \frac{24 }{\ell_j} \Big)^{2r_k\ell_j}(2r_k\ell_j)!\sum_{ \substack{ n_j=\square \\ \Omega(n_j) = 2r_k\ell_j \\ p|n_j \implies  p\in P_j}} \frac{1}{\sqrt{n_j}}\frac{1 }{w(n_j)}\prod_{p | n_j}\Big ( \frac p{p+1} \Big ) \ll \Big( \frac{24 }{\ell_j} \Big)^{2r_k\ell_j} \frac {(2r_k\ell_j)!}{ (r_k\ell_j)!}\Big (\sum_{p \in P_j}\frac 1{p} \Big )^{r_k\ell_j}.
\end{align*}
   We apply \eqref{boundsforsumoverp}  and \eqref{Stirling} to estimate the right side expression above to see that for some constant $B_1(k)$ depending on $k$ only,
 \begin{align*}
 &  \sum \text{square term in } {\mathcal Q}_j(d)  \times \text{a corresponding factor } \\\\
\ll & B_1(k)^{\ell_j}e^{-r_k\ell_j\log (r_k\ell_j)}\Big (\sum_{p \in P_j}\frac 1{p} \Big )^{r_k\ell_j}
\ll  B_1(k)^{\ell_j}e^{-r_k\ell_j\log (r_k\ell_j)}e^{r_k\ell_j \log (2\ell_j/N)} \\
\ll & e^{-\ell_j}\exp \Big (\frac {(2k)^2}{2}\sum_{p \in P_j}\frac 1{p}  \Big ).
\end{align*}

  Combining the above with \eqref{maintermbound} and \eqref{sqinN}, we see that the right side expression in \eqref{upperboundprodofN} is
\begin{align*}
\begin{split}
\ll & X \prod^R_{j=1} \Big (1+e^{-\ell_j} \Big ) \times \exp \Big (\frac {\big(2k \big )^2}{2}\sum_{p \in P_j}\frac 1p+\sum_{p \in P_j}\frac {1}{p^{2}}\Big ) \ll X ( \log X  )^{\frac {(2k))^2}{2}},
\end{split}
\end{align*}
  where we deduce the last estimation above using Lemma \ref{RS}. The assertion of the proposition now follows from this.

\vspace*{.5cm}

\noindent{\bf Acknowledgments.} P. G. is supported in part by NSFC grant 11871082.

\bibliography{biblio}
\bibliographystyle{amsxport}

\vspace*{.5cm}

\end{document}